\documentclass[twoside,leqno,10pt, A4]{amsart}
\usepackage{amsfonts}
\usepackage{amsmath}
\usepackage{amscd}
\usepackage{amssymb}
\usepackage{amsthm}
\usepackage{amsrefs}
\usepackage{latexsym}
\usepackage{mathrsfs}
\usepackage{bbm}
\usepackage{enumerate}
\usepackage{graphicx}

\usepackage{amsfonts}
\usepackage{amsmath}
\usepackage{amscd}
\usepackage{amssymb}
\usepackage{amsthm}
\usepackage{amsrefs}
\usepackage{latexsym}
\usepackage{mathrsfs}
\usepackage{bbm}
\usepackage{amscd}
\usepackage{amssymb}
\usepackage{amsthm}
\usepackage{amsrefs}
\usepackage{latexsym}
\usepackage{mathrsfs}
\usepackage{bbm}
\usepackage{enumerate}
\usepackage{graphicx}
\usepackage{color}
\begin{document}

\newtheorem{theorem}[subsection]{Theorem}
\newtheorem{proposition}[subsection]{Proposition}
\newtheorem{lemma}[subsection]{Lemma}
\newtheorem{corollary}[subsection]{Corollary}
\newtheorem{conjecture}[subsection]{Conjecture}
\newtheorem{prop}[subsection]{Proposition}
\numberwithin{equation}{section}
\newcommand{\mr}{\ensuremath{\mathbb R}}
\newcommand{\mc}{\ensuremath{\mathbb C}}
\newcommand{\dif}{\mathrm{d}}
\newcommand{\intz}{\mathbb{Z}}
\newcommand{\ratq}{\mathbb{Q}}
\newcommand{\natn}{\mathbb{N}}
\newcommand{\comc}{\mathbb{C}}
\newcommand{\rear}{\mathbb{R}}
\newcommand{\prip}{\mathbb{P}}
\newcommand{\uph}{\mathbb{H}}
\newcommand{\fief}{\mathbb{F}}
\newcommand{\majorarc}{\mathfrak{M}}
\newcommand{\minorarc}{\mathfrak{m}}
\newcommand{\sings}{\mathfrak{S}}
\newcommand{\fA}{\ensuremath{\mathfrak A}}
\newcommand{\mn}{\ensuremath{\mathbb N}}
\newcommand{\mq}{\ensuremath{\mathbb Q}}
\newcommand{\half}{\tfrac{1}{2}}
\newcommand{\f}{f\times \chi}
\newcommand{\summ}{\mathop{{\sum}^{\star}}}
\newcommand{\chiq}{\chi \bmod q}
\newcommand{\chidb}{\chi \bmod db}
\newcommand{\chid}{\chi \bmod d}
\newcommand{\sym}{\text{sym}^2}
\newcommand{\hhalf}{\tfrac{1}{2}}
\newcommand{\sumstar}{\sideset{}{^*}\sum}
\newcommand{\sumprime}{\sideset{}{'}\sum}
\newcommand{\sumprimeprime}{\sideset{}{''}\sum}
\newcommand{\sumflat}{\sideset{}{^\flat}\sum}
\newcommand{\shortmod}{\ensuremath{\negthickspace \negthickspace \negthickspace \pmod}}
\newcommand{\V}{V\left(\frac{nm}{q^2}\right)}
\newcommand{\sumi}{\mathop{{\sum}^{\dagger}}}
\newcommand{\mz}{\ensuremath{\mathbb Z}}
\newcommand{\leg}[2]{\left(\frac{#1}{#2}\right)}
\newcommand{\muK}{\mu_{\omega}}
\newcommand{\thalf}{\tfrac12}
\newcommand{\lp}{\left(}
\newcommand{\rp}{\right)}
\newcommand{\Lam}{\Lambda_{[i]}}
\newcommand{\lam}{\lambda}
\def\L{\fracwithdelims}
\def\om{\omega}
\def\pbar{\overline{\psi}}
\def\phis{\phi^*}
\def\lam{\lambda}
\def\lbar{\overline{\lambda}}
\newcommand\Sum{\Cal S}
\def\Lam{\Lambda}
\newcommand{\sumtt}{\underset{(d,2)=1}{{\sum}^*}}
\newcommand{\sumt}{\underset{(d,2)=1}{\sum \nolimits^{*}} \widetilde w\left( \frac dX \right) }

\newcommand{\hf}{\tfrac{1}{2}}
\newcommand{\af}{\mathfrak{a}}
\newcommand{\Wf}{\mathcal{W}}

\newtheorem{mylemma}{Lemma}
\newcommand{\intR}{\int_{-\infty}^{\infty}}

\theoremstyle{plain}
\newtheorem{conj}{Conjecture}
\newtheorem{remark}[subsection]{Remark}

\makeatletter
\def\widebreve{\mathpalette\wide@breve}
\def\wide@breve#1#2{\sbox\z@{$#1#2$}%
     \mathop{\vbox{\m@th\ialign{##\crcr
\kern0.08em\brevefill#1{0.8\wd\z@}\crcr\noalign{\nointerlineskip}%
                    $\hss#1#2\hss$\crcr}}}\limits}
\def\brevefill#1#2{$\m@th\sbox\tw@{$#1($}%
  \hss\resizebox{#2}{\wd\tw@}{\rotatebox[origin=c]{90}{\upshape(}}\hss$}
\makeatletter

\title[Bounds for moments of quadratic Dirichlet $L$-functions of prime-related moduli]{Bounds for moments of quadratic Dirichlet $L$-functions
of prime-related moduli}

\author{Peng Gao and Liangyi Zhao}

\begin{abstract}
 In this paper, we study the $k$-th moment of central values of the family of quadratic Dirichlet $L$-functions of moduli $8p$, with $p$ ranging over odd primes.  Assuming the truth of the generalized Riemann hypothesis, we establish sharp upper and lower bounds for the $k$-th power moment of these $L$-values for all real $k \geq 0$.
\end{abstract}

\maketitle



\section{Introduction}
\label{sec 1}

 Moments of families of $L$-functions at the central value are important subjects in analytical number theory due to their rich arithmetic applications. In \cite{Jutila}, M. Jutila initiated the study on the first and second moment of the family of quadratic Dirichlet $L$-functions $L( 1/2, \chi_d)$ for $d$ running over fundamental discriminants, where $\chi_d=\leg {d}{\cdot}$ is the Kronecker symbol. Asymptotic formulae for these moments were given in \cite{Jutila}. In the same paper, Jutila also obtained an asymptotic formula for the first moment of the family of quadratic Dirichlet $L$-functions of prime moduli $L(1/2, \chi_p)$, for primes $p$ satisfying various congruence conditions. This resolved a conjecture raised earlier by D. Goldfeld and C. Viola in \cite{G&V}. In \cite{B&P}, S. Baluyot and K. Pratt  obtained an asymptotic formula for the second moment of $L(\half, \chi_p)$ under the generalized Riemann hypothesis (GRH). They also obtained sharp upper and lower bounds for the second moment unconditionally as well as for the third moment under the assumption that $L(1/2, \chi_n) \geq 0$ for certain integers $n$. The function field analogue of moments of quadratic Dirichlet $L$-functions of prime moduli has been studied by J. Andrade and J. P. Keating in \cite{Andrade&Keating} and by H. M. Bui and A. Florea in \cite{B&F2020}. \newline

    Based on results from Random Matrix Theory,  J. P. Keating and N. C. Snaith  conjectured in \cite{Keating-Snaith02} that for all $k \geq 0$,
\begin{align*}
\begin{split}
  \sum_{|d| \leq X}L \left( \frac{1}{2}, \chi_d \right)^k \sim C_k X(\log X)^{k(k+1)/2}
\end{split}
\end{align*}
  as $X \rightarrow \infty$, where the numbers $C_k$ are explicit constants.  It is expected that the above relations continue to hold (with different constants), if one considers instead
\begin{align*}
\begin{split}
  \sum_{p \leq X} (\log p)L \left( \frac{1}{2}, \chi_p \right)^k.
\end{split}
\end{align*}

 The aim of this paper is to establish sharp bounds of the right order of magnitude for the above expression. Due to some technical reasons, we consider $\chi_{8p}$ instead of $\chi_{p}$ for odd primes $p$ and we assume the truth of GRH throughout.  In particular as a consequence of GRH, $L(1/2, \chi_{8p}) \geq 0$ for all odd primes $p$. \newline

  We shall establish upper and lower bounds separately. Our lower bound result is giving in the following.
\begin{theorem}
\label{thmlowerbound}
    Assume the truth of RH for $\zeta(s)$ and GRH for $L(s, \chi_{8p})$ for all odd primes $p$. We have for large $X$
   and all real numbers $k \geq 0$,
\begin{align*}
\begin{split}
    \sum_{\substack{2< p \leq X}} ( \log p) L \left( \frac{1}{2}, \chi_{8p} \right)^k  \gg_k & X(\log X)^{k(k+1)/2}.
\end{split}
\end{align*}
\end{theorem}

  The proof of Theorem \ref{thmlowerbound} relies on the lower bounds principal developed by W. Heap and K. Soundararajan in \cite{H&Sound} towards establishing sharp lower bounds for moments for families of $L$-functions at the central value. In this process, we shall actually need to make use of a method of K. Soundararajan in \cite{Sound2009} together with its refinement by A. J. Harper in \cite{Harper} to derive sharp upper bounds for moments of families of $L$-functions under GRH. The same approach also leads to our sharp upper bounds as follows.
\begin{theorem}
\label{thmupperbound}
   Assume the truth of RH for $\zeta(s)$ and GRH for $L(s, \chi_{8p})$ for all odd primes $p$. We have for large $X$
   and all real numbers $k \geq 0$,
\begin{align*}
\begin{split}
    \sum_{\substack{2< p \leq X}} ( \log p) L \left( \frac{1}{2}, \chi_{8p} \right)^k   \ll_k & X(\log X)^{k(k+1)/2}.
\end{split}
\end{align*}
\end{theorem}

Theorems \ref{thmlowerbound} and \ref{thmupperbound} can be combined to readily produce the following result on the order of magnitude of the family of $L$-functions under our consideration.
\begin{corollary}
\label{thmorderofmag}
   Assume the truth of RH for $\zeta(s)$ and GRH for $L(s, \chi_{8p})$ for all odd primes $p$. We have for large $X$
   and all real numbers $k \geq 0$,
\begin{align*}
\begin{split}
   \sum_{\substack{2< p \leq X}} ( \log p)L \left( \frac{1}{2}, \chi_{8p} \right)^k \asymp_k &
   X(\log X)^{k(k+1)/2}.
\end{split}
\end{align*}
\end{corollary}

\section{Preliminaries}
\label{sec 2}

   We gather first a few auxiliary results required in our proofs. In the rest of the paper, we shall reserve the symbols $p$, $q$
for primes.

\subsection{Sums over primes}
\label{sec2.4}

  We note the following result on various summations over prime numbers.
\begin{lemma}
\label{RS} Let $x \geq 2$. There is a constant $c$ such that
$$
\sum_{p\le x} \log p= x + O\big(x \exp(-c\sqrt{\log x})\big).
$$
 Also, there is a constant $b$ such that
$$
\sum_{p\le x} \frac{1}{p} = \log \log x + b+ O\Big(\frac{1}{\log x}\Big).
$$
\end{lemma}

The estimations in the above lemma are given by \cite[Theorem 6.9]{MVa1} and \cite[Theorem 2.7]{MVa1}, respectively.  There are, of course, better results under GRH.  But the unconditional results in Lemma~\ref{RS} are sufficient for our purpose.

\subsection{The approximate functional equation}
\label{sect: apprfcneqn}

  We recall the following approximate functional equation concerning $L(1/2, \chi_{8p})$ given in \cite[Lemma 2.2]{sound1}.
\begin{lemma}[Approximate functional equation]
\label{lem:AFE}
  For any odd prime  $p$, we have
\begin{align*}
\begin{split}
 L\left( \frac{1}{2}, \chi_{8d} \right) = & 2\sum^{\infty}_{\substack{n=1}} \frac{\chi_{8p}(n)}{\sqrt{n}} V
\left(\frac{ n}{\sqrt{d}} \right),
\end{split}
\end{align*}
 where for any real number $t>0$,
\begin{align}
\label{eq:Vdef}
 V(t) = \frac{1}{2 \pi i} \int\limits\limits_{(2)}  \left(\frac{8}{\pi}\right)^{s/2}
\left ( \frac {\Gamma(s/2+1/4)}{\Gamma(1/4)} \right ) t^{-s} \frac {\dif s}{s}.
\end{align}
\end{lemma}

\subsection{Upper bound for $\log |L(1/2, \chi_{8p})|$ }

  We cite a result from \cite[Proposition 3]{Harper} on an upper bound of $\log |L(\half, \chi_{8p})|$ in terms of a sum over primes.
\begin{lemma}
\label{lem: logLbound}
Assume RH for $\zeta(s)$ and GRH for $L(s,\chi_{8p})$ for all odd primes $p$. We have for $x \geq 2$ and $p \leq X$ for a large number
$X$,
\begin{align}
\label{equ:3.3'}
\begin{split}
 & \log  \left| L(\tfrac{1}{2}, \chi_{8p}) \right| \leq  \sum_{\substack{  q \leq x }} \frac{\chi_{8p} (q)}{q^{1/2+1/\log x}}
 \frac{\log (x/q )}{\log x} + \frac{1}{2} \log \log x+\frac{\log X}{\log x} + O(1) .
\end{split}
 \end{align}
\end{lemma}

\subsection{Smoothed character sums}
\label{smoothsum}

   For any positive odd integer $c$, we write $\psi_{c}$ for the Dirichlet character with $\psi_{c}(n) = \leg {n}{c} $ for $n \in \mz$. We also define $\delta_{c=\square}=1$ if $c$ is a perfect square and $\delta_{c=\square}=0$ otherwise.  In the remainder of the paper, let $\Phi$ denote a smooth, non-negative function compactly supported on $[1/2,5/2]$ satisfying $\Phi(x) =1$ for $x\in [1,2]$. The Mellin transform of $\Phi(x)$ is wirtten as ${\widehat \Phi}(s)$ and recall that for any complex number $s$,
\begin{equation*}
{\widehat \Phi}(s) = \int\limits_{0}^{\infty} \Phi(x)x^{s}\frac {\dif x}{x}.
\end{equation*}
   We have the following result on the smoothed quadratic character sums.
\begin{lemma}
\label{lemma logd}
Assume GRH. Let $c$ be a positive odd integer and $\Phi(X)$ be a smooth function fitting the above descriptions. Then for any $\varepsilon>0$,
\begin{equation} \label{wsum}
 \sum_{(p,2)=1} (\log p) \chi_{8p}(c) \Phi \left( \frac {p}X \right) =  \delta_{c=\square}\widehat{\Phi}(1)X+O \left( X^{1/2+\varepsilon}\log
 \log (c+2) \right).
 \end{equation}
\end{lemma}
\begin{proof}
  Due to the compact support of $\Phi$, we get
\begin{align*}
 \sum_{(p,2)=1} (\log p) \chi_{8p}(c) \Phi \left( \frac {p}X \right) =& \sum_{n}  \chi_{8n}(c) \Lambda(n) \Phi \left( \frac {n}X \right)
 +O \left(\sum_{\substack{p^j \leq X^{1+\varepsilon}, \ j \geq 2 }} (\log p)\Phi \left( \frac {p^j}X \right)   \right ).
\end{align*}

Now Lemma \ref{RS} gives
\begin{align}
\label{ppowerest}
 \sum_{\substack{p^j \leq X^{1+\varepsilon}, \ j \geq 2}}(\log p)\Phi \left( \frac {p^j}X \right) \ll
 X^{\varepsilon}\sum_{\substack{p \leq X^{1/2+\varepsilon}}} \log p \ll X^{1/2+\varepsilon}.
\end{align}

   Next, we apply Mellin inversion to obtain that
\begin{align*}
\sum_{n}  \chi_{8n}(c) \Lambda(n) \Phi \left( \frac {n}X \right)
=& -\frac {\chi_8(c)}{2\pi i}\int\limits_{(2)} \frac {L'(s, \psi_c)}{L(s, \psi_c)} \widehat{\Phi}(s)X^s \dif s.
\end{align*}
Recall here that $\psi_c$ denotes the Dirichlet character $\leg {n}{c} $ for $n \in \mz$. \newline

  We evaluate the above integral by shifting the line of integration to $\Re(s)=1/2+\varepsilon$.  There is a pole at $s=1$ with residue
  $-\widehat{\Phi}(1)X$ only if $c$ is a perfect square.  The integration on the line $\Re(s)=1/2+\varepsilon$ can be estimated as $O(X^{1/2+\varepsilon}\log \log (c+2))$ by using the rapid decay of $\widehat{\Phi}$ on the vertical line and the estimate (see
  \cite[Theorem 5.17]{iwakow}) that under GRH, we have for $\Re(s) \geq 1/2+\varepsilon$,
\begin{align}
\label{Lderbound}
  \frac {L'(s, \psi_c)}{L(s, \psi_c)}  \ll \log\log \big ((c+2)(1+|s|)\big).
\end{align}
  The expression given in \eqref{wsum} now follows.
\end{proof}

\section{Proof of Theorem \ref{thmlowerbound}}
\label{sec 2'}

\subsection{The lower bound principle}

    We assume that $X$ is large throughout the proof.  Upon dividing $q$  into dyadic
    blocks and replacing $k$ by $2k$, we see that in order to prove Theorem \ref{thmlowerbound}, it suffices to show that for $k \geq 0$,
\begin{align*}
   \sum_{(p,2)=1}(\log p)L(\tfrac{1}{2},\chi_{8p})^{2k}\Phi\left( \frac pX \right) \gg X(\log X)^{k(2k+1)},
\end{align*}
  where $\Phi$ is given in Section \ref{smoothsum}. \newline

   As the case $k=0$ is trivial and the case $k=1/2$ above essentially follows from \cite[Theorem 3]{Jutila} (with minor changes in the proof) due to M. Jutila, we shall assume that $0<k \neq 1/2$ throughout the proof.  We also note that in remainder of the paper, unless otherwise specified, the implicit constants involved in estimations using $\ll$ or the big-$O$ notations depend on $k$ only and are uniform with respect to $\chi$. We further recall the usual convention that the empty product is defined to be $1$. \newline

   We follow the ideas of A. J. Harper in \cite{Harper} to define for a large number $M$ depending on $k$ only,
$$ \alpha_{0} = \frac{\log 2}{\log X}, \;\;\;\;\; \alpha_{i} = \frac{20^{i-1}}{(\log\log X)^{2}} \;\;\; \mbox{for all} \; i \geq 1, \quad
\mathcal{I} = \mathcal{I}_{k,X} = 1 + \max\{i : \alpha_{i} \leq 10^{-M} \} . $$

The above notations and Lemma \ref{RS} yield that that for $X$ large enough,
\begin{align}
\label{sump1}
 \mathcal{I} \leq \log\log\log X , \;\;\;\;\; \alpha_{1} = \frac{1}{(\log\log X)^{2}} , \;\;\;\;\; \sum_{p \leq X^{1/(\log\log X)^{2}}}
\frac{1}{p} \leq \log\log X.
\end{align}

  Also, for $1 \leq i \leq \mathcal{I}-1$ and $X$ large enough,
\begin{align}
\label{sumpj}
\mathcal{I}-i \leq \frac{\log(1/\alpha_{i})}{\log 20} , \;\;\;\;\; \sum_{X^{\alpha_{i}} < p \leq X^{\alpha_{i+1}}} \frac{1}{p}
 = \log \alpha_{i+1} - \log \alpha_{i} + o(1) = \log 20 + o(1) \leq 10.
\end{align}

  Combining \eqref{sump1} and \eqref{sumpj}, we obtain
\begin{align*}
   \sum_{X^{\alpha_{i-1}} < p \leq X^{\alpha_{i}}}\frac 1{p} \leq \frac{100}{10^{3M/4}}\alpha^{-3/4}_i, \quad 1\leq i  \leq \mathcal{I}.
\end{align*}

For any real numbers $x$, $y$ with $y \geq 0$, we set
\begin{align} \label{E}
  E_{y}(x) = \sum_{j=0}^{2\lceil y \rceil} \frac {x^{j}}{j!}.
\end{align}
  We then define for any real number $\alpha$ and any $1\leq i  \leq \mathcal{I}$,
\begin{align*}
 {\mathcal P}_i(p)=&  \sum_{X^{\alpha_{i-1}} < q \leq X^{\alpha_{i}}}  \frac{\chi_{8p} (q)}{\sqrt{q}}, \quad {\mathcal N}_i(p, \alpha) = E_{e^2k\alpha^{-3/4}_i} \Big (\alpha {\mathcal P}_i(p) \Big ), \quad  {\mathcal N}(p, \alpha)=  \prod^{\mathcal{I}}_{i=1} {\mathcal N}_i(p, \alpha).
\end{align*}

  Note that each ${\mathcal N}_i(p,\alpha)$ is a short Dirichlet polynomial of length at most $X^{2\alpha_{i}\lceil e^2k\alpha^{-3/4}_i \rceil}$. By taking $X$ large enough, we have that
\begin{align}
\label{exponentbound}
 \sum^{\mathcal{I}}_{i=1} 2\alpha_{i}\lceil e^2k\alpha^{-3/4}_i \rceil \leq \lceil  4e^2k10^{-M/4}\rceil.
\end{align}
   It follows that ${\mathcal N}(p, \alpha)$ is also a short Dirichlet polynomial of length at most $X^{\lceil  4e^2k10^{-M/4}\rceil}$. \newline

  In the proof of Theorem \ref{thmlowerbound}, we need the following bounds for expressions involving with various  ${\mathcal N}(p, \alpha)$ .
\begin{lemma}
\label{lemNbounds}
 With the notations as above, we have for $0<k<1/2$ and $1 \leq i \leq \mathcal{I}$,
\begin{align}
\label{est0}
\begin{split}
{\mathcal N}_i(p, 2k-1)^{\frac {2(2-3k)}{1-2k}} {\mathcal N}_i(p, 2-2k)^{2} \le
{\mathcal N}_i(p, 2k)  \left( 1+ O\big(e^{-e^2k\alpha^{-3/4}_i} \big)
\right) + {\mathcal Q}^{r_k}_i(p, k).
\end{split}
\end{align}

  We also have  for $k >1/2$ and $1 \leq i \leq \mathcal{I}$,
\begin{align}
\label{est0'}
\begin{split}
\mathcal{N}_i(p,
 2k-1)^{\frac {2k}{2k-1}}  \le
{\mathcal N}_i(p, 2k)  \left( 1+ O\big(e^{-e^2k\alpha^{-3/4}_i} \big)
\right) + {\mathcal Q}^{r_k}_i(p, k).
\end{split}
\end{align}
  Here the implied constants in \eqref{est0} and \eqref{est0'} are absolute, and we define
$$
{\mathcal Q}_i(p,k) =\left( \frac{12 \max (9, 36k^2 ) {\mathcal P}_i(p)}{\lceil e^2k\alpha^{-3/4}_i \rceil} \right)^{ 2\lceil e^2k\alpha^{-3/4}_i \rceil},
$$
  with $r_k=3+\lceil 2(2-3k)/(1-2k) \rceil$ for $0<k<1/2$ and $r_k=1+\lceil 2k/(2k-1) \rceil$ for $k>1/2$.
\end{lemma}
\begin{proof}
  As in the proof of \cite[Lemma 3.4]{Gao2021-3}, we have for $|z| \le aK/20$ with $0<a \leq 2$,
\begin{align}
\label{Ebound}
\Big| \sum_{r=0}^K \frac{z^r}{r!} - e^z \Big| \le \frac{|z|^{K}}{K!} \le \Big(\frac{a e}{20}\Big)^{K}.
\end{align}
  By taking $z=\alpha {\mathcal P}_i(p), K=2\lceil e^2k\alpha^{-3/4}_i \rceil$ and $a=\min (|\alpha|, 2 )$ in \eqref{Ebound}, we see that for any $\alpha' \geq |\alpha|$, if
  \[ |{\mathcal P}_i(p)| \le \lceil e^2k\alpha^{-3/4}_i \rceil/(10(1+\alpha')), \]
  then
\begin{align*}
{\mathcal N}_i(p, \alpha)=& \exp ( \alpha{\mathcal P}_i(p) )\left( 1+  O \left( \exp ( |\alpha {\mathcal P}_i(p)| ) \left( \frac{a e}{20} \right)^{2 e^2k\alpha^{-3/4}_i} \right) \right )\\
= & \exp ( \alpha {\mathcal P}_i(p)  ) \left( 1+  O\left( a e^{-e^2k\alpha^{-3/4}_i} \right) \right).
\end{align*}

  We apply the above estimates to ${\mathcal N}_i(p, 2k-1), \ {\mathcal N}_i(p, 2-2k)$ and ${\mathcal N}_i(p, 2k)$ by taking $\alpha'=2$ above to see that when $0<k<1/2$ and  $|{\mathcal P}_i(p)| \le \lceil e^2k\alpha^{-3/4}_i \rceil/30$, then
\begin{align} \label{est1}
\begin{split}
{\mathcal N}_i(p, 2k-1)^{\frac {2(2-3k)}{1-2k}} {\mathcal N}_i(p, 2-2k)^{2}
=& \exp ( 2k  {\mathcal P}_i(p)  ) \left( 1+ O\big( e^{-e^2k\alpha^{-3/4}_i} \big) \right) \\
=& {\mathcal N}_i(p, 2k) \left( 1+ O\big(e^{-e^2k\alpha^{-3/4}_i} \big)
\right) .
\end{split}
\end{align}

  Similarly, by taking $\alpha'=2k-1$, we see that when $k>1/2$ and $|{\mathcal P}_i(p)| \le \lceil e^2k\alpha^{-3/4}_i \rceil/(20k)$, then
\begin{align}
\label{est1'}
\begin{split}
\mathcal{N}_i(p,
 2k-1)^{\frac {2k}{2k-1}}  \le
{\mathcal N}_i(p, 2k)  \left( 1+ O\big(e^{-e^2k\alpha^{-3/4}_i} \big)
\right).
\end{split}
\end{align}

 On the other hand, when $|{\mathcal P}_i(p)| \ge \lceil e^2k\alpha^{-3/4}_i \rceil/(10(1+\alpha'))$, we have that
\begin{align}
\label{4.2}
\begin{split}
{\mathcal N}_i(p, \alpha) \le \sum_{r=0}^{2\lceil e^2k\alpha^{-3/4}_i \rceil} \frac{|\alpha{\mathcal P}_i(p)|^r}{r!} & \le
|(\alpha'+1) {\mathcal P}_i(p)|^{2\lceil e^2k\alpha^{-3/4}_i \rceil} \sum_{r=0}^{2\lceil e^2k\alpha^{-3/4}_i \rceil} \Big( \frac{10(1+\alpha')}{2\lceil e^2k\alpha^{-3/4}_i \rceil}\Big)^{\lceil e^2k\alpha^{-3/4}_i \rceil-r} \frac{1}{r!}  \\
&   \le \Big( \frac{12(\alpha'+1)^2 |{\mathcal
P}_i(p)|}{\lceil e^2k\alpha^{-3/4}_i \rceil}\Big)^{2\lceil e^2k\alpha^{-3/4}_i \rceil} .
\end{split}
\end{align}

The last expression in \eqref{4.2} enables us to deduce that when $0<k<1/2$ and $|{\mathcal P}_i(p)| \ge \lceil e^2k\alpha^{-3/4}_i \rceil/30$,
\begin{align}
\label{est2}
\begin{split}
{\mathcal N}_i(p, 2k-1)^{\frac {2(2-3k)}{1-2k}} {\mathcal N}_i(p, 2-2k)^{2} \leq
{\mathcal Q}^{r_k}_i(p,k).
\end{split}
\end{align}

   Moreover, we set $\alpha'=2k-1$ in \eqref{4.2} to deduce that when $k>1/2$ and $|{\mathcal P}_i(p)| \ge \lceil e^2k\alpha^{-3/4}_i \rceil/(20k)$,
\begin{align}
\label{est2'}
\begin{split}
 \mathcal{N}_i(p,
 2k-1)^{\frac {2k}{2k-1}}  \le {\mathcal Q}^{r_k}_i(p,k).
\end{split}
\end{align}

   The assertion of the lemma now follows from \eqref{est1}, \eqref{est1'}, \eqref{est2} and \eqref{est2'}.
\end{proof}

  Next, we state the needed lower bounds principle of W. Heap and K. Soundararajan in
  \cite{H&Sound} for our situation.
\begin{lemma}
\label{lem1}
 With notations as above, we have for $0<k<1/2$ and $c=(2/k-3)^{-1}$,
\begin{align}
\label{basicboundkbig}
\begin{split}
  \sum_{(p,2)=1} & (\log p)  L \left( \frac{1}{2},\chi_{8p} \right)  \mathcal{N}(p, 2k-1) \Phi\leg{p}{X}  \\
 \ll &\left( \sum_{(p,2)=1}(\log p)L \left( \frac{1}{2}, \chi_{8p} \right)^{2k} \Phi \left( \frac pX \right) \right)^{c/(2k)} \left( \sum_{(p,2)=1}(\log p)L \left( \frac{1}{2}, \chi_{8p} \right)^2 \mathcal{N}(p, 2k-2)\Phi \left( \frac pX \right) \right)^{(1-c)/2} \\
 & \hspace*{2cm} \times \left(
 \sum_{(p,2)=1}(\log p) \prod^{{\mathcal I}}_{i=1} \big ( {\mathcal N}_i(p, 2k)+ {\mathcal Q}^{r_k}_i(p,k) \big )\Phi\leg{p}{X}  \right)^{(1+c)/2-c/(2k)}.
\end{split}
\end{align}
  For $k > 1/2$, we have
\begin{align}
\label{basicboundkbig1'}
\begin{split}
 \sum_{(p,2)=1} (\log p) & L \left( \frac{1}{2}, \chi_{8p} \right) \mathcal{N}(p,2k-1)  \Phi \left( \frac pX \right)  \\
 \leq & \Big ( \sum_{(p,2)=1}(\log p)L \left( \frac{1}{2}, \chi_{8p} \right)^{2k} \Phi \left( \frac pX \right)\Big )^{1/(2k)}\Big (\sum_{(p,2)=1}(\log p) \prod^{{\mathcal I}}_{i=1} \big ( {\mathcal N}_i(p, 2k)+ {\mathcal Q}^{r_k}_i(p,k) \big )\Phi \left( \frac pX \right) \Big)^{1-1/(2k)}.
\end{split}
\end{align}
\end{lemma}
\begin{proof}
  We note from its definition $E_{\ell}(x)$ is an even degree polynomial for any positive integer $\ell$. We then proceed as in \cite[Section 2]{Gao2021-3} to get that for any real number $\alpha$,
\begin{align} \label{Nprodbound}
 \mathcal{N}(p, \alpha)\mathcal{N}(p, -\alpha)  \geq 1.
\end{align}
 We deduce from \eqref{Nprodbound} that
\begin{align*}
\begin{split}
 \sum_{(p,2)=1}(\log p) & L \left( \frac{1}{2}, \chi_{8p} \right) \mathcal{N}(p,2k-1)  \Phi \left( \frac pX \right)  \\
 \leq & \sum_{(p,2)=1}(\log p)L \left( \frac{1}{2}, \chi_{8p} \right)^{c}\cdot L \left( \frac{1}{2}, \chi_{8p} \right)^{1-c} \mathcal{N}(p, 2k-2)^{(1-c)/2}  \cdot
 \mathcal{N}(p,2k-1) \mathcal{N}(p, 2-2k)^{(1-c)/2} \Phi \left( \frac pX \right).
\end{split}
\end{align*}

  We now apply H\"older's inequality with exponents $2k/c, 2/(1-c), ((1+c)/2-c/(2k))^{-1}$ to the
  last sum above.   We easily confirm, using the definition of $c$, that these exponents are all at least $1$. This leads to
\begin{align} \label{basicbound1}
\begin{split}
\sum_{(p,2)=1} (\log p) & L \left( \frac{1}{2}, \chi_{8p} \right) \mathcal{N}(p, 2k-1)   \Phi \left( \frac pX \right) \\
 \leq & \left( \sum_{(p,2)=1}(\log p)L \left( \frac{1}{2}, \chi_{8p} \right)^{2k} \Phi \left( \frac pX \right)\right)^{c/(2k)} \left( \sum_{(p,2)=1}(\log p)L \left( \frac{1}{2}, \chi_{8p} \right)^2 \mathcal{N}(p, 2k-2)\Phi \left( \frac pX \right) \right)^{(1-c)/2} \\
 & \hspace*{2cm} \times \left(\sum_{(p,2)=1}(\log p) \mathcal{N}(p, 2k-1)^{\frac {2(2-3k)}{1-2k}}\mathcal{N}(p, 2-2k)^{2} \Phi \left( \frac pX \right)
 \right)^{(1+c)/2-c/(2k)}.
\end{split}
\end{align}
The estimation given in \eqref{basicboundkbig} then follows from the above by applying the estimation in \eqref{est0} in the last sum of \eqref{basicbound1}. \newline

 Similarly, when $k > 1/2$, H\"older's inequality with the exponents $2k$, $2k/(2k-1)$ yeilds
\begin{align} \label{basicbound2}
\begin{split}
 \sum_{(p,2)=1}(\log p) & L \left( \frac{1}{2}, \chi_{8p} \right) \mathcal{N}(p,2k-1)  \Phi \left( \frac pX \right)  \leq \sum_{(p,2)=1}(\log p) \left| L \left( \frac{1}{2}, \chi_{8p} \right) \right|
 \mathcal{N}(p,2k-1)  \Phi \left( \frac pX \right) \\
 \leq & \left( \sum_{(p,2)=1}(\log p)L \left( \frac{1}{2}, \chi_{8p} \right)^{2k} \Phi \left( \frac pX \right) \right)^{1/(2k)} \left(\sum_{(p,2)=1}(\log p) \mathcal{N}(p,
 2k-1)^{\frac {2k}{2k-1}} \Phi \left( \frac pX \right) \right)^{1-1/(2k)}.
\end{split}
\end{align}
The estimation given in \eqref{basicboundkbig1'} then follows from above by applying the estimation in \eqref{est0'} in the last sum of \eqref{basicbound2}. This completes the proof.
\end{proof}

  It follows from the above lemma that in order to establish Theorem \ref{thmlowerbound}, it suffices to prove the following three propositions.
\begin{proposition}
\label{Prop4}
  With notations as above, we have for $k > 0$,
\begin{align}
\label{L1estmation}
 \sum_{(p,2)=1}(\log p) L \left( \frac{1}{2}, \chi_{8p} \right){\mathcal N}(p, 2k-1)\Phi\Big(\frac{p}{X}\Big) \gg X (\log X)^{ \frac {(2k)^2+1}{2}}.
\end{align}
\end{proposition}

\begin{proposition}
\label{Prop6}
  With notations as above, we have for $0<k<1/2$,
\begin{align*}
\sum_{(p,2)=1}(\log p) L \left( \frac{1}{2}, \chi_{8p} \right)^2{\mathcal N}(p, 2k-2)\Phi\Big(\frac{p}{X}\Big)  \ll X ( \log X  )^{\frac {(2k)^2+2}{2}}.
\end{align*}
\end{proposition}

\begin{proposition}
\label{Prop5}
  With notations as above, we have for $k >0$, $k \neq 1/2$,
\begin{align*}
\sum_{(p,2)=1}(\log p) \prod^{\mathcal{I}}_{i=1} \big ( {\mathcal N}_i(p, k)^2+ {\mathcal Q}^{r_k}_i(p,k)^2 \big )\Phi\leg{p}{X}  \ll & X ( \log X  )^{\frac
{(2k)^2}{2}}.
\end{align*}
\end{proposition}

  Notice that Proposition \ref{Prop5} can be established in a manner similar to \cite[Proposition 2.2]{Gao2021-3}, upon using Lemma \ref{lemma logd} in our situation. Thus, it remains to establish the other two propositions. In the course of the proof of Proposition \ref{Prop4}, we need the following result on the twisted first moment of quadratic Dirichlet $L$-functions of prime moduli.
\begin{proposition}
\label{twistedfirstmoment}
  With notations as above and let $X$ be a large real number.  Suppose $\Phi$ is a function satisfying the conditions given in Section~\ref{smoothsum}.  Further let $\ell$ be a fixed positive integer and write $\ell$ uniquely as
  $\ell=\ell_1\ell^2_2$ with $\ell_1$ square-free, we have
\begin{align*}
 \sum_{(p,2)=1}(\log p)L \left( \frac{1}{2}, \chi_{8p} \right)\chi_{8p}(\ell) \Phi\leg{p}{X} = C_1 \frac {X}{\sqrt{\ell_1}} \log \Big ( \frac {X}{\ell_1} \Big
 )+C_2 \frac {X}{\sqrt{\ell_1}} + O(X^{7/8 + \varepsilon}\ell_1^{-1/4+\varepsilon}+X^{1/2+\varepsilon}),
\end{align*}
  where $C_1$, $C_2$ are explicit constants depending on $\Phi$ only.
\end{proposition}

\begin{proof}
  We apply the approximate functional equation given in Lemma \ref{lem:AFE} to see that
\begin{equation*} \label{splitting}
\mathcal{M} :=\sum_{(p,2)=1} (\log p) L \left( \frac{1}{2}, \chi_{8p} \right) \chi_{8p}(\ell) \Phi\leg{p}{X}
=2
\sum^{\infty}_{m=1} \frac {1}{m^{1/2}}\sum_{(p,2)=1}(\log p)\chi_{8p}(m\ell )V \left( \frac {m}{\sqrt{p}} \right)  \Phi\left(\frac{p}{X}\right).
\end{equation*}

   Due to the compact support of $\Phi$, we get that
\begin{align*}
 \sum_{(p,2)=1}(\log p)\chi_{8p}(m\ell )V \left( \frac {m}{\sqrt{p}} \right) \Phi\left(\frac{p}{X}\right) =& \sum^{\infty}_{n=1}  \chi_{8n}(m\ell )
 \Lambda(n) V \left( \frac {m}{\sqrt{n}} \right) \Phi \left( \frac {n}X \right)
 +O \left(\sum_{\substack{p^j \leq X^{1+\varepsilon}, j \geq 2 }} (\log p)\Phi \left( \frac {p^j}X \right)   \right ).
\end{align*}

  By \eqref{ppowerest}, the $O$-term above is $\ll X^{1/2+\varepsilon}$. To deal with the first term on the right-hand side of the above, we apply Mellin
   inversion to obtain that
\begin{align}
\label{twistedcharsum}
\sum^{\infty}_{n=1}  \chi_{8n}(m\ell ) \Lambda(n) V \left( \frac {m}{\sqrt{n}} \right) \Phi \left( \frac {n}X \right)
=& -\frac {\chi_8(m\ell )}{2\pi i}\int\limits_{(2)} \frac {L'(s, \psi_{m\ell})}{L(s, \psi_{m\ell})} \widehat{f}(s)X^s \dif s,
\end{align}
  where
\begin{equation*}
\widehat{f}(s) = \int\limits_0^{\infty} V \left(\frac{m}{\sqrt{X}} x^{-1/2}\right) \Phi(x) x^{s-1} \dif x.
\end{equation*}

The function $V$ satisfies the following bounds (see \cite[Lemma 2.1]{sound1}).
\[ V(x) = 1 + O \left( x^{1/2-\varepsilon} \right), \quad \mbox{as} \quad x \to 0 \]
and for large $x$ and integers $j \geq 0$.
\[ V^{(j)}(x) \ll_j \exp (- x^2/2) . \]
Now repeated integration by parts, together with the above bounds for $V$, yields
\begin{equation*}
 \widehat{f}(s) \ll (1 + |s|)^{-E} \left( 1 + m/\sqrt{X} \right)^{-E},
\end{equation*}
for any $\Re(s) >0$ and any integer $E>0$. \newline

   We evaluate the integral in \eqref{twistedcharsum} by shifting the line of integration to the line $\Re(s)=1/2+\varepsilon$ to encounter a pole at
   $s=1$ with residue $-\widehat{f}(1)X$ only if $m\ell$ is a perfect square.  The integral over $\Re(s)=1/2+\varepsilon$ can be estimated
   to be $O(X^{1/2+\varepsilon}\log \log (m\ell+2))$ using the rapid decay of $\widehat{f}$ on the vertical line and \eqref{Lderbound}. \newline

  To deal with the contribution from the residues, we deduce via the expression for $V$ in \eqref{eq:Vdef} that
\begin{align*}
 \widehat{f}(1) = \int_0^{\infty} V\left(\frac{m}{\sqrt{X}} x^{-1/2} \right) \Phi(x) dx = \frac{1}{2 \pi i} \int\limits_{(2)}
 \leg{\sqrt{X}}{m}^s\widehat{\Phi} \left( 1+\frac s2 \right)   \left(\frac{8}{\pi}\right)^{s/2} \frac {\Gamma(s/2+1/4)}{\Gamma(1/4)}
 \frac {\dif s}{s}.
\end{align*}	

Note that $m\ell$ is a perfect square if and only if $m$ is a square multiple of $\ell_1$.  So we may replace $m$ by $\ell_1 m^2$ and get that the contribution from the poles to $\mathcal{M}$ is
\begin{align} \label{M0integral}
  \frac {2X}{\sqrt{\ell_1}}\frac{1}{2 \pi i} \int\limits_{(2)} \widehat{\Phi} \left( 1+\frac s2 \right)   \left(\frac{8}{\pi}\right)^{s/2} \frac
  {\Gamma(s/2+1/4)}{\Gamma(1/4)}\leg{\sqrt{X}}{\ell_1}^s \left( 1-\frac 1{2^{1+2s}} \right) \zeta(1+2s) \frac {\dif s}{s}.
\end{align}

  We evaluate the integral in \eqref{M0integral} by moving the contour of integration to $-1/4 + \varepsilon$, crossing a pole at $s=0$ only.
  To estimate the integral on the new line, we apply the convexity bound for $\zeta(s)$ given in \cite[(5.20)]{iwakow}  to see that, when
  $\Re(s)=1/2+\varepsilon$,
\begin{equation*}
 \zeta(s) \ll (1+|s|)^{1/4+\varepsilon}.
\end{equation*}

  The residue of the pole in the above process can be easily computed and this leads to the proof of the proposition.
\end{proof}

\subsection{Proof of Proposition \ref{Prop4}}
\label{sec 4}

Let $w(n)$ be the multiplicative function such that $w(p^{\beta}) = \beta!$ for prime powers $p^{\beta}$ and let $\Omega(n)$ denote the number of distinct primes dividing $n$.  Let $b_i(n), 1 \leq i \leq {\mathcal{I}}$ be functions such that $b_i(n)=0$ or $1$ and that $b_i(n)=1$ only when $n$ is composed of at most $2\lceil e^2k\alpha^{-3/4}_i \rceil$ primes, all from the interval $(X^{\alpha_{i-1}}, X^{\alpha_i}]$. \newline

These notations allow us to write for any real number $\alpha$,
\begin{equation}
\label{5.1}
{\mathcal N}_i(p, \alpha) = \sum_{n_i} \frac{1}{\sqrt{n_i}} \frac{\alpha^{\Omega(n_i)}}{w(n_i)}  b_i(n_i) \chi_{8p}(n_i), \quad 1\le i\le {\mathcal{I}}.
\end{equation}

We now apply Proposition~\ref{twistedfirstmoment} and \eqref{5.1} to evaluate the left side expression in \eqref{L1estmation}. In this process, we may ignore the error term in Proposition~\ref{twistedfirstmoment} as ${\mathcal N}(p, 2k-1)$ is a short Dirichlet polynomial by \eqref{exponentbound}.  In the same decomposition of $l$ in the statement of Proposition~\ref{twistedfirstmoment}, $n_i$ can be uniquely written as $n_i=(n_i)_1(n_i)_{2}^2$ with $(n_i)_1, (n_i)_2 \in \intz$ and $(n_i)_{1}$ square-free.  Thus main term contribution leads to
\begin{align*}
\sum_{(p,2)=1}(\log p) & L \left( \frac{1}{2}, \chi_{8p} \right){\mathcal N}(p, 2k-1)\Phi\Big(\frac{p}{X}\Big) \\
 \gg & X   \sum_{n_1, \cdots, n_{\mathcal{I}} }  \Big( \prod_{i=1}^{\mathcal{I}} \frac{1}{\sqrt{n_i(n_i)_{1}}}
\frac{(2k-1)^{\Omega(n_i)}}{w(n_i) } b_i(n_i)   \Big) \Big (\log \Big ( \frac {\sqrt{X}}{(n_1)_1 \cdots (n_{\mathcal{I}})_1} \Big )+C_2 \Big ).
\end{align*}

   An expression similar to the right side expression above already appears in the proof of \cite[Propsition 2.1]{Gao2021-3}. It follows from
   the treatment there that the right side expression above is $\gg X (\log X)^{((2k)^2+1)/2}$.
 This completes the proof of the proposition.

\subsection{Proof of Proposition \ref{Prop6}}
\label{sec: proof of Prop 6}

We begin by establishing some weaker estimations on the upper bounds for moments of the quadratic families of Dirichlet $L$-functions under consideration.  Let $X$ be a large number and $\mathcal{N}(V, X)$ be the number of Dirichlet characters $\chi_{8p}$  such that $2 < p \leq X$ and $\log L(1/2,\chi_{8p}) \geq V+\half \log \log X$.  We then have that
\begin{align}
\label{momentint}
\begin{split}
  \sum_{\substack{2<p \leq X}} (\log p) L \left( \frac{1}{2}, \chi_{8p} \right)^{k} & \leq -(\log X) \int\limits_{-\infty}^{+\infty}\operatorname{exp} \Big( kV+\frac k2 \log \log X \Big) \dif \mathcal{N}(V,X) \\
 & = k(\log X)^{k/2+1} \int\limits_{-\infty}^{+\infty}\operatorname{exp}(kV)\mathcal{N}(V,X) \dif V.
\end{split}
\end{align}

   To estimate the last integral above, we note that by setting $x =\log X$ and summing over $p$ trivially in \eqref{equ:3.3'} to see that we may assume that $\mathcal{N}(V, X) \leq 6\log X/\log \log X$.  For $V < 10\sqrt{\operatorname{log}\operatorname{log}
  X}$, we use the trivial bounds $\mathcal{N}(V, X) \ll X/(\log X)$. For the remaining range of $V$, we note the following upper bounds for $\mathcal{N}(V, X)$ given in \cite[Section 4]{Sound2009}.
\begin{prop}
\label{propNbound}
Assume RH for $\zeta(s)$ and GRH for $L(s,\chi_{8p})$ for all odd primes $p$. We have for any fixed $k>0$,
 \begin{align*}
  \mathcal{N}(V,X)\ll \left\{
 \begin{array}[c]{ll}
\displaystyle{  \frac {X}{\log X} \exp \left(-\frac{V^2}{2 \log \log X}(1+o(1)) \right ), \quad 10\sqrt{\operatorname{log}\operatorname{log}
  X}\leq V \leq o((\log \log X)(\log \log \log X)) }, \\ \\
\displaystyle{  \frac {X}{\log X} \exp \left(-c V \log V \right), \quad o((\log \log X)(\log \log \log X)) \leq V \leq \frac{6\log X}{\log \log X} }.
\end{array}
\right .
 \end{align*}
\end{prop}

  Applying the above bounds for $\mathcal{N}(V, X)$ in \eqref{momentint} readily leads to the following weaker upper bounds for moments of the family of quadratic Dirichlet $L$-functions .
\begin{prop}
\label{prop: upperbound}
Assume RH for $\zeta(s)$ and GRH for $L(s, \chi_{8p})$ for all odd primes $p$.  For any positive real number $k$ and any $\varepsilon>0$, we
have for large $X$,
\begin{align*}
    \sum_{\substack{2<p \leq X}} (\log p) L \left( \frac{1}{2}, \chi_{8p} \right)^{k}  \ll_k & X(\log X)^{k(k+1)/2+\varepsilon}.
\end{align*}
\end{prop}

  Now, we take exponentials on both sides of \eqref{equ:3.3'} and deduce that
\begin{align}
\label{basicest}
\begin{split}
 & L \left( \frac{1}{2}, \chi_{8p} \right)^{k}  \ll
 \exp \left(k \left( \sum_{\substack{  q \leq x }} \frac{\chi_{8p} (q)}{q^{1/2+1/\log x}}
 \frac{\log (x/q)}{\log x} +\frac 12 \log \log x
 +\frac{\log X}{\log x}  \right)\right ).
\end{split}
 \end{align}
  To estimate the right side expression above, we denote
\begin{align*}
{\mathcal M}_{i,j}(p) =& \sum_{X^{\alpha_{i-1}} < q \leq X^{\alpha_{i}}}  \frac{\chi_{8p} (q)}{q^{\tfrac{1}{2}+1/(\log X^{\alpha_{j}})}}
\frac{\log (x^{\alpha_{j}}/q)}{\log x^{\alpha_{j}}}, \quad 1\leq i \leq j \leq \mathcal{I} .
\end{align*}

  We also define the following set:
\begin{align*}
 \mathcal{S}(0) =& \{ 0< p \leq X : |  {\mathcal M}_{1,l}(p)| > \alpha_{1}^{-3/4} \; \text{ for some } 1 \leq l \leq \mathcal{I} \} ,   \\
 \mathcal{S}(j) =& \{ 0< p \leq X : |  {\mathcal M}_{i,l}(p)| \leq  \alpha_{i}^{-3/4} \; \mbox{for all}\; 1 \leq i \leq j, \; \mbox{and} \; i \leq l \leq \mathcal{I}, \\
 & \hspace*{2cm} \text{but} \; |{\mathcal M}_{j+1,l}(p)| > \alpha_{j+1}^{-3/4} \; \text{ for some } j+1 \leq l \leq \mathcal{I} \} , \quad 1\leq j \leq \mathcal{I}, \\
 \mathcal{S}(\mathcal{I}) =& \{0< p \leq X : |{\mathcal M}_{i, \mathcal{I}}(p)| \leq \alpha_{i}^{-3/4} \; \mbox{for all} \; 1 \leq i \leq \mathcal{I}\}.
\end{align*}

   As
$$ \{ 0< p \leq X   \}=\bigcup_{j=0}^{ \mathcal{I}} \mathcal{S}(j), $$
 it remains to show that
\begin{align}
\label{sumovermj'}
  \sum_{j=0}^{\mathcal{I}}\sum_{p \in \mathcal{S}(j)} (\log p) L \left( \frac{1}{2}, \chi_{8p} \right)^{2}\mathcal{N}(p, 2k-2)\Phi \left( \frac pX \right)  \ll  X(\log X)^{((2k)^2+2)/2}.
\end{align}

   We let $W(t)$ be any non-negative smooth function that is supported on $(1/2-\varepsilon_1, 1+\varepsilon_1)$ for some fixed small $0<\varepsilon_1<1/2$ such that $W(t) \gg 1$ for $t \in (1/2, 1)$. We then notice that
\begin{align}
\label{S0est}
\begin{split}
\text{meas}(\mathcal{S}(0)) \ll \sum_{\substack{(p,2)=1 }}
\sum^{\mathcal{I}}_{l=1}
\Big ( \alpha^{3/4}_{1}{|\mathcal
M}_{1, l}(\chi)| \Big)^{2\lceil 1/(10\alpha_{1})\rceil }W \left( \frac pX \right) =
\sum^{\mathcal{I}}_{l=1}\sum_{\substack{(p,2)=1 }} \Big (
\alpha^{3/4}_{1}{|\mathcal
M}_{1, l}(\chi)| \Big)^{2\lceil 1/(10\alpha_{1})\rceil }W \left( \frac p X \right) .
\end{split}
\end{align}

   We now apply the estimates in \eqref{sump1} to evaluate the last sums in \eqref{S0est}.  Using an approach similar to the proof of Proposition~\ref{Prop4}, we get
\begin{align}
\label{S0bound}
\text{meas}(\mathcal{S}(0)) \ll &
\mathcal{I}X e^{-1/(6\alpha_{1})}\ll X e^{-(\log\log X)^{2}/10}  .
\end{align}

   We then deduce via the Cauchy-Schwarz inequality that
\begin{align} \label{LS0bound}
\begin{split}
\sum_{\chi \in  \mathcal{S}(0)} & (\log p) L \left( \frac{1}{2}, \chi_{8p} \right)^{2}\mathcal{N}(p, 2k-2)\Phi \left( \frac pX \right)  \\
\leq &  \Big ( (\log X) \text{meas}(\mathcal{S}(0)) \Big )^{1/4} \left(
\sum_{\substack{(p,2)=1 }} (\log p)L \left( \frac{1}{2}, \chi_{8p} \right)^{8}\Phi \left( \frac pX \right) \right)^{1/4} \Big ( \sum_{\substack{(p,2)=1 }} (\log p)\mathcal{N}(p, 2k-2)^{2} \Phi \left( \frac pX \right) \Big )^{1/2}.
\end{split}
\end{align}

  Similar to the proof of Proposition \ref{Prop5}, we have that
\begin{align} \label{N2k2bound}
 \sum_{\substack{(p,2)=1 }} (\log p)\mathcal{N}(p, 2k-2)^{2} \Phi \left( \frac pX \right) \ll X ( \log X  )^{(2(2k-2))^2/2}.
\end{align}

Also, by Proposition \ref{prop: upperbound} with $\varepsilon=1$, we have that
\begin{align}
\label{L8bound}
\sum_{\substack{(p,2)=1 }} (\log p)L \left( \frac{1}{2}, \chi_{8p} \right)^{8}\Phi \left( \frac pX \right) \leq \sum_{\substack{0<p\leq 4X}} (\log p)L(\half, \chi_{8p})^{8} \ll
X(\log X)^{37}.
\end{align}

  Applying the bounds given in \eqref{S0bound}, \eqref{N2k2bound} and \eqref{L8bound} in \eqref{LS0bound}, we deduce that
\begin{align*}
\sum_{\chi \in  \mathcal{S}(0)} (\log p) L \left( \frac{1}{2}, \chi_{8p} \right)^{2}\mathcal{N}(p, 2k-2)\Phi \left( \frac pX \right)  \ll X(\log X)^{k(k+1)/2}.
\end{align*}

  The above estimation implies that it remains to consider the cases $j \geq 1$ in \eqref{sumovermj'}. When $p \in \mathcal{S}(j)$, we set $x=X^{\alpha_j}$ in \eqref{basicest} to see
  that
\begin{align}
\label{Lkbound}
\begin{split}
 & L \left( \frac{1}{2}, \chi_{8p} \right)^{k} \ll (\log X)^{k/2} \exp \Big(\frac {k}{\alpha_j} \Big) \exp \Big (
 k \sum^j_{i=0}{\mathcal M}_{i,j}(p) \Big ).
\end{split}
 \end{align}

   As we have $M_{i, j} \leq  \alpha^{-3/4}_i$ when $p \in \mathcal{S}(j)$, we argue as in the proof of \cite[Lemma 5.2]{Kirila} to see that
\begin{align}
\label{eMbound}
\begin{split}
\exp \Big ( k\sum^j_{i=1}{\mathcal M}_{i,j}(p)\Big ) \ll
\prod^j_{i=1}E_{e^2k\alpha^{-3/4}_i}(k{\mathcal M}_{i,j}(p)).
\end{split}
 \end{align}

   We then deduce from the description on $\mathcal{S}(j)$ that when $j \geq 1$,
\begin{align*}
\begin{split}
 \sum_{p \in \mathcal{S}(j)} (\log p) & L \left( \frac{1}{2}, \chi_{8p} \right)^{2}\mathcal{N}(p, 2k-2)\Phi(\frac pX)  \\
 \ll &  (\log X) \exp \left(\frac {2}{\alpha_j} \right)
 \sum^{ \mathcal{I}}_{l=j+1} \sum_{\substack{2<p \leq X}} \exp \Big ( 2 \sum^j_{i=1}{\mathcal M}_{i,j}(p)\Big )\mathcal{N}(p, 2k-2) \Big ( \alpha^{3/4}_{j+1}{\mathcal
M}_{j+1, l}(p)\Big)^{2\lceil 1/(10\alpha_{j+1})\rceil } \\
\ll & (\log X) \exp \left(\frac {2}{\alpha_j}\right)  \sum^{ \mathcal{I}}_{l=j+1}
\sum_{\substack{2<p \leq X}}
\prod^j_{i=1} E_{e^2k\alpha^{-3/4}_i}(2{\mathcal M}_{i,j}(p))E_{e^2k\alpha^{-3/4}_i}((2k-2){\mathcal P}_{i}(p)) \\
& \hspace*{1cm} \times E_{e^2k\alpha^{-3/4}_{j+1}}((2k-2){\mathcal P}_{j+1}(p)) \Big ( \alpha^{3/4}_{j+1}{\mathcal
M}_{j+1, l}(p)\Big)^{2\lceil 1/(10\alpha_{j+1})\rceil } \prod^{\mathcal{I}}_{i=j+2} E_{e^2k\alpha^{-3/4}_i}((2k-2){\mathcal P}_{i}(p)).
\end{split}
\end{align*}

We now apply \eqref{sumpj} and proceed as in the proofs of Proposition \ref{Prop4} to arrive at
\begin{align*}
\begin{split}
  \sum_{p \in \mathcal{S}(j)}(\log p) & L \left( \frac{1}{2}, \chi_{8p} \right)^{2}\mathcal{N}(p, 2k-2)\Phi \left( \frac pX \right)  \\
 \ll &  (\log X) \exp \Big(\frac {2}{\alpha_j} \Big) (\mathcal{I}-j)e^{-44/\alpha_{j+1}} \prod_{p \leq X^{\alpha_{\mathcal{I}}}}\Big  (1+\frac {(2k)^2}{2p}+O \Big( \frac1{p^2} \Big) \Big ) \ll  e^{-42/\alpha_{j+1}}X (\log X)^{((2k)^{2}+2)/2}.
\end{split}
\end{align*}

 As $20/\alpha_{j+1}=1/\alpha_j$, we conclude from the above that
\begin{align*} \begin{split}
 & \sum_{p \in \mathcal{S}(j)}(\log p) L \left( \frac{1}{2}, \chi_{8p} \right)^{2}\mathcal{N}(p, 2k-2)\Phi \Big( \frac pX \Big) \ll e^{-1/(10\alpha_{j})}X(\log X)^{((2k)^{2}+2)/2}.
\end{split}
 \end{align*}

   As the sum of the right side expression above over $j$ converges, we see that the above estimation implies \eqref{sumovermj'}
and this completes the proof of Proposition \ref{Prop6}.

\section{Proof of Theorem \ref{thmupperbound}}

  The proof of Theorem \ref{thmupperbound} is similar to that of Proposition \ref{Prop6}. Thus we shall be brief here. In keeping the notations in Section \ref{sec: proof of Prop 6}, we see that it suffices to show that for $k>0$,
\begin{align}
\label{sumovermj}
  \sum_{j=0}^{\mathcal{I}}\sum_{p \in \mathcal{S}(j)} (\log p) L \left( \frac{1}{2}, \chi_{8p} \right)^{k}  \ll  X(\log X)^{k(k+1)/2}.
\end{align}

  Once again we may only examine the case $j \geq 1$ here. We apply \eqref{Lkbound} and \eqref{eMbound} to see that when $j \geq 1$,
\begin{align*} \begin{split}
 \sum_{p \in \mathcal{S}(j)}(\log p) & L \left( \frac{1}{2}, \chi_{8p} \right)^{k} \\
\ll & (\log X)^{k/2}\exp \Big(\frac {k}{\alpha_j} \Big) \sum^{\mathcal{I}}_{l=j+1}
\sum_{\substack{2<p \leq X}}
\prod^j_{i=1} E_{e^2k\alpha^{-3/4}_i}(k{\mathcal M}_{i,j}(\chi))\Big ( \alpha^{3/4}_{j+1}{\mathcal
M}_{j+1, l}(\chi)\Big)^{2\lceil 1/(10\alpha_{j+1})\rceil } .
\end{split}
 \end{align*}

 Applying \ref{sumpj} and arguing as in the proof of Proposition \ref{Prop4}, we obtain that
\begin{align*}
\begin{split}
 \sum^{\mathcal{I}}_{l=j+1} \sum_{\substack{2<p \leq X}} &
\prod^j_{i=1} E_{e^2k\alpha^{-3/4}_i}(k{\mathcal M}_{i,j}(\chi)) \Big ( \alpha^{3/4}_{j+1}{\mathcal
M}_{j+1, l}(\chi)\Big)^{2\lceil 1/(10\alpha_{j+1})\rceil }  \\
\ll & X(\mathcal{I}-j)e^{-22k/\alpha_{j+1}} \prod_{p \leq X^{\alpha_j}}\Big  (1+\frac {k^2}{2p}+O \Big( \frac
1{p^2} \Big) \Big)  \ll  e^{-21k/\alpha_{j+1}}X (\log X)^{k^{2}/2}.
\end{split}
 \end{align*}

   We conclude from the above that
\[ \sum_{p \in \mathcal{S}(j)}(\log p) L \left( \frac{1}{2}, \chi_{8p} \right)^{k} \ll e^{-k/(20\alpha_{j})}X(\log X)^{k(k+1)/2}. \]

Summing over  $j$ now leads to \eqref{sumovermj} and this completes the proof of Theorem \ref{thmupperbound}.
\vspace*{.5cm}

\noindent{\bf Acknowledgments.}  P. G. is supported in part by NSFC grant 11871082 and L. Z. by the FRG grant PS43707 at the University of New South Wales (UNSW). The authors are grateful to the anonymous referee for his/her very careful reading of this manuscript and many helpful comments and suggestions.

\bibliography{biblio}
\bibliographystyle{amsxport}

\vspace*{.5cm}

\noindent\begin{tabular}{p{8cm}p{8cm}}
School of Mathematical Sciences & School of Mathematics and Statistics \\
Beihang University & University of New South Wales \\
Beijing 100191 China & Sydney NSW 2052 Australia \\
Email: {\tt penggao@buaa.edu.cn} & Email: {\tt l.zhao@unsw.edu.au} \\
\end{tabular}

\end{document}